\documentclass[11pt, notitlepage]{article}   

\usepackage{amsmath,amsthm,amsfonts}   

\usepackage{amscd}
\usepackage{amsfonts}
\usepackage{amssymb}
\usepackage{color}      
\usepackage{epsfig}
\usepackage{graphicx}           

\usepackage{tikz}

\theoremstyle{plain}
\newtheorem{theorem}{Theorem}[section]
\newtheorem{lemma}[theorem]{Lemma}
\newtheorem{corollary}[theorem]{Corollary}
\newtheorem{proposition}[theorem]{Proposition}
\newtheorem{observation}[theorem]{Observation}
\newtheorem{example}[theorem]{Example}

\theoremstyle{definition}

\newcommand{\sdim}{\textnormal{sdim}}
\newcommand{\diam}{\textnormal{diam}}
\newcommand{\code}{\textnormal{code}}
\newcommand{\ter}{\textnormal{ter}}

\def\finf{\mathop{{\rm I}\kern -.27 em {\rm F}}\nolimits}

\def\SR{\rm SR}

\newcommand{\rmv}[1]{}

\begin{document}

\title{Maker-Breaker Strong Resolving Game}

\author{{\bf{Cong X. Kang}}$^1$, {\bf{Aleksander Kelenc}}$^2$ and {\bf{Eunjeong Yi}}$^3$\\
{\small Texas A\&M University at Galveston, Galveston, TX 77553, USA}$^{1,3}$\\
{\small University of Maribor, FEECS Koro\v{s}ka cesta 46, 2000 Maribor, Slovenia}$^2$\\
{\small\em kangc@tamug.edu}$^1$; {\small\em aleksander.kelenc@um.si}$^2$; {\small\em yie@tamug.edu}$^3$}

\maketitle

\date{}

\begin{abstract}
Let $G$ be a graph with vertex set $V$. A set $S \subseteq V$ is a \emph{strong resolving set} of $G$ if, for distinct $x,y\in V$, there exists $z\in S$ such that either $x$ lies on a $y-z$ geodesic or $y$ lies on an $x-z$ geodesic in $G$. In this paper, we study maker-breaker strong resolving game (MBSRG) played on a graph by two players, Maker and Breaker, where the two players alternately select a vertex of $G$ not yet chosen. Maker wins if he is able to choose vertices that form a strong resolving set of $G$ and Breaker wins if she is able to prevent Maker from winning in the course of MBSRG. We denote by $O_{\rm SR}(G)$ the outcome of MBSRG played on $G$. We obtain some general results on MBSRG and examine the relation between $O_{\rm SR}(G)$ and $O_{\rm R}(G)$, where $O_{\rm R}(G)$ denotes the outcome of the maker-breaker resolving game of $G$. We determine the outcome of MBSRG played on some graph classes, including corona product graphs, Cartesian product graphs, and modular product graphs.
\end{abstract}




\section{Introduction}\label{sec_intro}

Erd\"{o}s and Selfridge~\cite{erdos} introduced the Maker-Breaker game played on an arbitrary hypergraph $H=(V,E)$ by two players, Maker and Breaker, who alternately select a vertex from $V$ not yet chosen in the course of the game. Maker wins the game if he can select all vertices of a hyperedge from $E$ and Breaker wins otherwise. For further reading on the topic, we refer to ~\cite{beck, hefetz}.
Maker-Breaker metric resolving game (MBRG) and Maker-Breaker distance-$k$ metric resolving games (MB$k$RG), respectively, played on a graph was introduced and studied in~\cite{mbrg} and~\cite{mbkrg}. MB$k$RG is a generalization of MBRG, where MB$k$RG equals MBRG when $k+1$ is at least the diameter of the graph being considered. For the fractionalization of MBRG, we refer to~\cite{FMBRG}. Other interesting two-player games played on a graph include cop and robber game~\cite{cop} and Maker-Breaker domination game~\cite{mbdg}, for examples. In this paper, we introduce and study Maker-Breaker strong resolving game (MBSRG) played on a graph, where MBSRG is a variant of MBRG.

We recall the notion of a (strong) resolving set and the (strong) metric dimension of a graph. Let $G$ be a finite, simple, undirected, and connected graph with vertex set $V(G)$ and edge set $E(G)$. For $x,y\in V(G)$, let $d_G(x, y)$ denote the minimum number of edges on a path connecting $x$ and $y$ in $G$; we drop the subscript if it is clear in context. A set $R \subseteq V(G)$ is a \emph{resolving set} of $G$ if, for distinct $x,y\in V(G)$, there exists a vertex $z\in R$ such that $d(x,z) \neq d(y,z)$. A set $S \subseteq V(G)$ is a \emph{strong resolving set} of $G$ if, for distinct $x,y\in V(G)$, there exists $z\in S$ such that either $x$ lies on a $y-z$ geodesic or $y$ lies on an $x-z$ geodesic in $G$. The \emph{metric dimension}, $\dim(G)$, of $G$ is the minimum cardinality among all resolving sets of $G$, and the \emph{strong metric dimension}, $\sdim(G)$, of $G$ is the minimum cardinality among all strong resolving sets of $G$. For an ordered set $S=\{u_1, u_2, \ldots, u_{k}\}$ and a vertex $v\in V(G)$, the metric code of $v$ with respect to $S$, denoted by $\code_S(v)$, is the $k$-vector $(d(v, u_1), d(v, u_2), \ldots, d(v, u_{k}))$. Seb\"{o} and Tannier~\cite{sebo} observed that, if $S$ is a strong resolving set, then the vectors $\{\code_S(v): v\in V(G)\}$ uniquely determine the graph $G$ (see~\cite{sdim_frac} for a proof for the claim) while providing two non-isomorphic graphs $H_1$ and $H_2$ with $V(H_1)=V(H_2)$ satisfying $\{\code_R(v): v\in V(H_1)\}=\{\code_R(v): v\in V(H_2)\}$ if $R$ is a resolving set for both $H_1$ and $H_2$. Harary and Melter~\cite{harary} and Slater~\cite{slater} independently introduced the concept of metric dimension, and Seb\"{o} and Tannier~\cite{sebo} introduced strong metric dimension.  Applications of metric dimension include robot navigation~\cite{landmarks}, network discovery and verification~\cite{network}, sonar~\cite{slater} and combinatorial optimization~\cite{sebo}, to name a few. It is known that determining the metric dimension and the strong metric dimension, respectively, of a general graph is an NP-hard problem; see~\cite{NP,landmarks} and~\cite{MMD}, respectively.

We recall one of the most important results regarding the strong metric dimension of a graph. For $v\in V(G)$, the \emph{open neighborhood} of $v$ is $N(v)=\{u \in V(G) : uv \in E(G)\}$. A vertex $u\in V(G)$ is \emph{maximally distant} from $v \in V(G)$ if $d(u,v) \ge d(w,v)$ for every $w\in N(u)$. If $u$ is maximally distant from $v$ and $v$ is maximally distant from $u$ in $G$, then we say that $u$ and $v$ are \emph{mutually maximally distant} in $G$ and we write $u\!$ MMD $\!v$. It was shown in~\cite{MMD} that if $x$ MMD $y$ in $G$, then $S \cap \{x,y\} \neq \emptyset$ for any strong resolving set $S$ of $G$. Based on~\cite{MMD} and~\cite{GSR}, the \emph{strong resolving graph} of $G$, denoted by $G_{\rm SR}$, has vertex set $V(G_{\rm SR})=\{x\in V(G): \exists y\in V(G) \mbox{ with } x \mbox{ MMD }y\}$ and $uv\in E(G_{\rm SR})$ if and only if $u$ MMD $v$ in $G$. A set $S \subseteq V(G)$ is a \emph{vertex cover} of $G$ if every edge of $G$ is incident with at least one vertex of $S$, and the \emph{vertex cover number}, $\tau(G)$, of $G$ the minimum cardinality among all vertex covers of $G$. It was shown in~\cite{MMD} that $\sdim(G)=\tau(G_{\rm SR})$.

Now, following~\cite{mbrg}, MBRG is played on a graph $G$ by two players, Maker and Breaker. Maker and Breaker alternately select (without missing their turn) a vertex of $G$ not yet chosen in the course of the game. M-game (B-game, respectively) denotes the game for which Maker (Breaker, respectively) plays first. Maker wins MBRG if he is able to select vertices that form a resolving set of $G$ in the course of the game; Breaker wins if Breaker can prevent Maker from winning the game. Replacing a resolving set by a strong resolving set, MBSRG is analogously defined in this paper. We denote by $O_{\rm R}(G)$ and $O_{\rm SR}(G)$, respectively, the outcome of MBRG and MBSRG played on $G$.

In a network $Z$ with rivals $M$ and $B$, suppose $M$ tries to install transmitters/sensors at some nodes in $Z$ that form a strong resolving set (enabling $M$ to map out the entire network $Z$) while $B$ tries to prevent $M$ from having a control over $Z$. Noting that there's no advantage for the second player in MBRG, it was observed in~\cite{mbrg} that there are three realizable outcomes for $O_{\rm R}(G)$. Analogously, there are three realizable outcomes for $O_{\rm SR}(G)$ as follows: (i) $O_{\rm SR}(G)=\mathcal{M}$ if Maker has a winning strategy for both the M-game and the B-game; (ii) $O_{\rm SR}(G)=\mathcal{B}$ if Breaker has a winning strategy for both the M-game and the B-game; (iii) $O_{\rm SR}(G)=\mathcal{N}$ if the first player has a winning strategy. In~\cite{mbrg,mbkrg}, the authors studied the minimum number of moves/steps/time for Maker (Breaker, respectively) to win the game provided Maker (Breaker, respectively) has a winning strategy in MBRG or MB$k$RG. Analogously, one can define the minimum number of moves/steps/time for each player to win or lose the MBSRG. However, in this paper, we will focus on the outcomes of MBSRG played on graphs.

The paper is organized as follows. In Section~\ref{sec_comparison}, we obtain some general results on the outcome of MBSRG and study the relation between $O_{\rm SR}(G)$ and $O_{\rm R}(G)$. In Section~\ref{sec_graphs}, we determine $O_{\rm SR}(G)$ for some graph classes, including corona product graphs, Cartesian product graphs, and modular product graphs.

We recall some additional terminology and notations. The \emph{diameter}, $\diam(G)$, of $G$ is $\max\{d(u,w): u,w\in V(G)\}$. The \emph{degree} of a vertex $v\in V(G)$, denoted by $\deg_G(v)$, is $|N(v)|$; a \emph{leaf} (or an \emph{end-vertex}) is a vertex of degree one, a \emph{support vertex} is a vertex that is adjacent to a leaf, a \emph{major vertex} is a vertex of degree at least three, and a \emph{universal vertex} is a vertex of degree $|V(G)|-1$. For $X\subseteq V(G)$, $G[X]$ denotes the subgraph of $G$ induced by the vertices in $X$; i.e., for $x,y\in X$, $xy\in E(G[X])$ if and only if $xy\in E(G)$. Two distinct vertices $u$ and $w$ in a graph are called \emph{adjacent twins} if $N(u) \cup\{u\}=N(w)\cup\{w\}$ and they are called \emph{non-adjacent twins} if $N(u)=N(w)$; $u$ and $w$ are called \emph{twins} if they are either adjacent twins or non-adjacent twins. Let $P_n$, $C_n$, and $K_n$, respectively, denote the path, the cycle, and the complete graph on $n$ vertices, and let $K_{s,t}$ denote the complete bi-partite graph with two parts of sizes $s$ and $t$. For any positive integer $k$, let $[k]$ denote the set $\{1,2,\ldots, k\}$.


\section{General results, and $O_{\rm SR}(G)$ versus $O_{\rm R}(G)$}\label{sec_comparison}

In this section, we obtain some general results on MBSRG and examine the relationship between $O_{\rm SR}(G)$ and $O_{\rm R}(G)$. We recall some known results on the strong metric dimension of graphs. A set $X\subseteq V(G)$ is called a \emph{twin-free clique} if no two vertices of $X$ are twins in $G$ and $G[X]$, the subgraph induced by $X$, forms a clique. The \emph{twin-free clique number}, $\overline{\omega}(G)$, of a graph $G$ is the maximum cardinality among all twin-free cliques in $G$.

\begin{theorem}\emph{\cite{MMD}}\label{thm_mmd}
For any connected graph $G$, $\sdim(G)=\tau(G_{\rm SR})$.
\end{theorem}

\begin{theorem}\emph{\cite{yi_sdim}}\label{thm_characterization}
Let $G$ be a connected graph of order $n \ge 2$. Then $1\le \sdim(G)\le n-1$, and we have the following:
\begin{itemize}
\item[(a)] $\sdim(G)=1$ if and only if $G=P_n$,
\item[(b)] $\sdim(G)=n-1$ if and only if $G=K_n$,
\item[(c)] for $n\ge 4$, $\sdim(G)=n-2$ if and only if $\diam(G)=2$ and $\overline{\omega}(G)=2$.
\end{itemize}
\end{theorem}

It is known that metric dimension is not a monotone parameter on subgraph inclusion (see~\cite{MB, broadcast}). We show that strong metric dimension is not a monotone parameter with respect to subgraph inclusion, but it is monotone with respect to strong resolving graphs. For analogous results on the fractional strong metric dimension of graphs, we refer to~\cite{sdimF_products}.

\begin{proposition}
\begin{itemize}
\item[(a)] The ratio $\frac{\sdim(H)}{\sdim(G)}$ is arbitrarily large for suitable choice of graphs $H \subset G$.
\item[(b)] If $H_{\rm SR} \subseteq G_{\rm SR}$, then $\sdim(H) \le \sdim(G)$.
\end{itemize}
\end{proposition}

\begin{proof}
(a) We denote by $X \square Y$ the Cartesian product of two graphs $X$ and $Y$. Let $H=C_{2m}$ and $G=P_m\square K_2$, where $m\ge3$. Then $H \subset G$, $H_{\rm SR}=mK_2$ and $G_{\rm SR}=2K_2$. So, $\sdim(H)=m$ and $\sdim(G)=2$. Thus, $\frac{\sdim(H)}{\sdim(G)}=\frac{m}{2} \rightarrow \infty$ as $m\rightarrow \infty$.

(b) Let $H_{\rm SR} \subseteq G_{\rm SR}$. Let $W$ be a minimum strong resolving set of $G$; then $W$ is a minimum vertex cover of $G_{\rm SR}$. Since an edge in $H_{\rm SR}$ is an edge in $G_{\rm SR}$, $W\cap V(H_{\rm SR})$ is a vertex cover of $H_{\rm SR}$. By Theorem~\ref{thm_mmd}, $\sdim(H)=\tau(H_{\rm SR}) \le \tau(G_{\rm SR})=\sdim(G)$.~\hfill
\end{proof}

As stated in Section~\ref{sec_intro}, two vertices $u$ and $w$ in $G$ are twins if $N(u)-\{w\}=N(w)-\{u\}$. Hernando et al.~\cite{Hernando} observed that the twin relation is an equivalence relation on $V(G)$ and, under it, each (twin) equivalence class induces either a clique or an independent set. The following observation is often useful in establishing a lower bound for the (strong) metric dimension of a graph.

\begin{observation}\label{obs_twin}
Let $u$ and $w$ be distinct members of the same twin equivalence class of a graph $G$.
\begin{itemize}
\item[(a)] \emph{\cite{Hernando}} For any resolving set $R$ of $G$, $R \cap \{u,w\} \neq \emptyset$.
\item[(b)] \emph{\cite{sdim_frac}} For any strong resolving set $S$ of $G$, $S \cap \{u,w\} \neq\emptyset$.
\end{itemize}
\end{observation}

We recall the following result on the outcome of MBRG. See \cite{mbrg,mbkrg} for definitions of pairing resolving set and quasi-pairing resolving set in the metric context; their definitions in the strong metric context, which are essentially identical, are given in the ensuing paragraph.

\begin{proposition}\label{outcome_r}
Let $G$ be a connected graph.
\begin{itemize}
\item[(a)] \emph{\cite{mbrg}} If $G$ admits a pairing resolving set, then $O_{\rm R}(G)=\mathcal{M}$.
\item[(b)] \emph{\cite{mbkrg}} If $G$ admits a quasi-pairing resolving set, then $O_{\rm R}(G)\in\{\mathcal{M}, \mathcal{N}\}$.
\item[(c)] \emph{\cite{FMBRG}} If $G$ has $k \ge 0$ twin equivalence class(es) of cardinality $2$ and exactly one twin equivalence class of cardinality $3$ with $\dim(G)=k+2$, then $O_{\rm R}(G)=\mathcal{N}$.
\item[(d)] \emph{\cite{mbrg}} If $G$ has a twin equivalence class of cardinality at least $4$, then $O_{\rm R}(G)=\mathcal{B}$.
\item[(e)] \emph{\cite{mbrg}} If $G$ has two distinct twin equivalence classes of cardinality $3$, then $O_{\rm R}(G)=\mathcal{B}$.
\end{itemize}
\end{proposition}

Analogous to the concept of a pairing dominating set (see~\cite{mbdg}) and a pairing (distance-$k$) resolving set (see~\cite{mbrg, mbkrg}), we define a pairing strong resolving set as follows. Given a set $X=\cup_{i=1}^{\beta}\{\{u_i, w_i\}\}$ such that $\cup X \subseteq V(G_{\rm SR})$ and $|\cup X|=2\beta$, $X$ is called a \emph{pairing strong resolving set} of $G$ if a set $Z \subseteq V(G_{\rm SR})$ is a strong resolving set of $G$ whenever the condition ($\dag$) $|Z|=\beta$ and $Z \cap\{u_i, w_i\}\neq\emptyset$ for each $i\in[\beta]$ holds; in this case, we will also call $X$ a \emph{pairing vertex cover} of $G_{\rm SR}$. Analogous to the concept of a quasi-pairing distance-$k$ resolving set in~\cite{mbkrg}, we define a quasi-pairing strong resolving set as follows. Let $X$ be as given. If there exists a vertex $v\in V(G_{\rm SR})-\cup X$ such that $Z\cup \{v\}$ is a strong resolving set of $G$ whenever $Z \subseteq V(G_{\rm SR})$ satisfies condition ($\dag$), then $X$ is called a \emph{quasi-pairing strong resolving set} of $G$; in this case, we will also call $X$ a \emph{quasi-pairing vertex cover} of $G_{\rm SR}$. 

Now, we obtain some general results on the outcome of MBSRG, which will be used in determining $O_{\rm SR}(G)$ in Section~\ref{sec_graphs}. Let $\mathbb{Z}^+$ denote the set of positive integers. Let $\Delta(H)$ denote the maximum degree of a graph $H$ and, for $u\in V$, let $u^c$ denote $V-\{u\}$. Next, we obtain our main theorem on the outcome of MBSRG.

\begin{theorem}\label{main_result}Let $G$ be a connected graph.
\begin{itemize}
\item[(a)] Suppose $G$ admits a pairing strong resolving set (equivalently, if $G_{\rm SR}$ admits a pairing vertex cover), then $O_{\rm SR}(G)=\mathcal{M}$.
\item[(b)] Suppose $G$ admits a quasi-pairing strong resolving set (equivalently, if $G_{\rm SR}$ admits a quasi-pairing vertex cover), then $O_{\rm SR}(G)\in\{\mathcal{M}, \mathcal{N}\}$.
\item[(c)] Suppose $\Delta(G_{\SR})\ge 2$, then $O_{\rm SR}(G)\in\{\mathcal{N}, \mathcal{B}\}$.
\item[(d)] Suppose $\Delta(G_{\SR}[u^c])\geq 2$ for every $u\in V(G_{\SR})$; then $O_{\rm SR}(G)=\mathcal{B}$.
\end{itemize}
\end{theorem}

\begin{proof}
Parts (a) and (b) follow immediately from the definitions of a pairing strong resolving set and a quasi-pairing strong resolving set, respectively. For part (c), suppose $G_{\SR}$ contains a vertex $w$ of degree at least two. In the $B$-game, Breaker can pick $w$ and one of its neighbors of $w$ as her first two moves, thus denying Maker a vertex cover of $G_{\SR}$. The following is a proof of part (d). Breaker, playing second, picks as her first move a vertex $v$ having at least two neighbors in $G_{\SR}[u^c]$, where $u\in V(G_{\SR})$ is the first move of Maker. Breaker can pick as her second move a neighbor of $v$ in $G_{\SR}$ regardless of the two moves made by Maker, thus denying Maker a vertex cover of $G_{\SR}$.~\hfill
\end{proof}

\begin{corollary}\label{outcome_sr}
Let $G$ be a connected graph.
\begin{itemize}
\item[(a)] $O_{\rm SR}(G)=\mathcal{M}$ if and only if $G_{\rm SR} = xK_2$ for $x\in \mathbb{Z}^+$.
\item[(b)] Let $H\in\{C_3, P_3, P_4, P_5\}$. If $G_{\rm SR} =H \cup xK_2$ for $x\in \mathbb{Z}^+ \cup\{0\}$, then $O_{\rm SR}(G)=\mathcal{N}$.
\item[(c)] Let $H^*\in\{K_m, C_m, P_{m+2}, xC_3 \cup yP_3\}$, where $m\ge4$ and $x+y\geq 2$. If $G_{\rm SR} \supseteq H^*$, then $O_{\rm SR}(G)=\mathcal{B}$.
\end{itemize}
\end{corollary}

\begin{proof}
(a) First, let $G_{\rm SR} = xK_2$ for $x\in \mathbb{Z}^+$ with $V(G_{\rm SR})=\{u_1, u_2, \ldots, u_{2x}\}$ and $E(G_{\rm SR})=\cup_{i=1}^{x}\{u_{2i-1}u_{2i}\}$. Since $\cup_{i=1}^{x}\{\{u_{2i-1},u_{2i}\}\}$ is a pairing vertex cover of $G_{\rm SR}$, we have $O_{\rm SR}(G)=\mathcal{M}$ by Theorem~\ref{main_result}(a). To see the other implication, if $\Delta(G_{\SR})\ge 2$, then $O_{\rm SR}(G)\in\{\mathcal{N}, \mathcal{B}\}$ by Theorem~\ref{main_result}(c). The condition $O_{\rm SR}(G)=\mathcal{M}$ thus implies that $\Delta(G_{\SR})=1$.

(b) First, suppose $G_{\rm SR} = C_3 \cup xK_2$ or $G_{\rm SR} = P_3 \cup xK_2$, where $x\in \mathbb{Z}^+ \cup\{0\}$, with vertex set $V(G_{\rm SR})=\{w_1, w_2, w_3\} \cup\{u_1, u_2, \ldots, u_{2x}\}$. Let  $E(C_3)=\{w_1w_2, w_2w_3, w_3w_1\}$, $E(P_3)=\{w_1w_2, w_2w_3\}$, and $E(xK_2)=\cup_{i=1}^{x}\{u_{2i-1}u_{2i}\}$. Notice $X=\{\{w_1, w_3\}\} \cup (\cup_{i=1}^{x}\{\{u_{2i-1}, u_{2i}\}\})$ is a quasi-pairing vertex cover of $G_{\rm SR}$ as $Z \cup \{w_2\}$ is a strong resolving set for $G$ satisfying condition ($\dag$); thus, $O_{\rm SR}(G)\in\{\mathcal{M},\mathcal{N}\}$ by Theorem~\ref{main_result}(b). In the B-game, Breaker can select $w_2$ as her first move and a neighbor of $w_2$ as her second move, thus denying Maker a vertex cover. We thus have $O_{\rm SR}(G)=\mathcal{N}$. 

Second, suppose $G_{\rm SR} = P_4 \cup xK_2$, where $x\in \mathbb{Z}^+ \cup\{0\}$, with vertex set $V(G_{\rm SR})=\{w_1, w_2, w_3, w_4\} \cup\{u_1, u_2, \ldots, u_{2x}\}$ and edge set $E(G_{\rm SR})=\{w_1w_2, w_2w_3, w_3w_4\} \cup (\cup_{i=1}^{x}\{u_{2i-1}u_{2i}\})$. Notice $X=\{\{w_3, w_4\}\} \cup (\cup_{i=1}^{x}\{\{u_{2i-1}, u_{2i}\}\})$ is a quasi-pairing vertex cover of $G_{\rm SR}$ as $Z \cup \{w_2\}$ is a strong resolving set for $G$ satisfying condition ($\dag$), rendering $O_{\rm SR}(G)\in\{\mathcal{M},\mathcal{N}\}$ by Theorem~\ref{main_result}(b). In the B-game, Breaker can select $w_2$ as her first move and a neighbor of $w_2$ as her second move, thus denying Maker a vertex cover. We thus have $O_{\rm SR}(G)=\mathcal{N}$. 

Third, suppose $G_{\rm SR} = P_5 \cup xK_2$, where $x\in \mathbb{Z}^+ \cup\{0\}$. In the M-game, Maker picks as his first move the central vertex $u$ of the $P_5$. Then (\emph{the derived strong resolving graph}) $G_{\SR}[u^c]=(x+2)K_2$, and Maker wins the B-game played on the derived strong resolving graph by part (a) of the present corollary. In the B-game, Breaker picks as her first move any degree-two vertex $u$ of the $P_5$, and she can pick as her second move a neighbor of $u$, thus denying Maker a vertex cover of $G_{\SR}$.

(c) This part, as well as other consequences, follows immediately from part (d) of Theorem~\ref{main_result}. ~\hfill
\end{proof}

Next, we provide a realization result on the outcome of MBSRG. We recall some terminology and notations. Fix a tree $G$. A leaf $\ell$ is called a \emph{terminal vertex} of a major vertex $v$ if $d(\ell,v) < d(\ell,w)$ for every other major vertex $w$ in $G$. The terminal degree, $\ter_G(v)$, of a major vertex $v$ is the number of terminal vertices of $v$ in $G$. An \emph{exterior major vertex} is a major vertex that has a positive terminal degree. We denote by $\sigma(G)$ and $ex(G)$, respectively, the number of leaves and the number of exterior major vertices of $G$.

\begin{proposition}\label{realization}
For each $n\ge 4$, there exist graphs $G_1$, $G_2$, $G_3$ of order $n$ such that $O_{\rm SR}(G_1)=\mathcal{M}$, $O_{\rm SR}(G_2)=\mathcal{N}$ and $O_{\rm SR}(G_3)=\mathcal{B}$.
\end{proposition}

\begin{proof}
Let $n\ge 4$. Let $G_1$ be an $n$-path, $G_2$ be a tree with $ex(G_2)=1$ and $\sigma(G_2)=3$, and $G_3$ be a complete graph on $n$ vertices. Then $(G_1)_{\rm SR}= K_2$, $(G_2)_{\rm SR} = K_3$ and $(G_3)_{\rm SR} = K_n$, where $n\ge4$. By Corollary~\ref{outcome_sr}, we have $O_{\rm SR}(G_1)=\mathcal{M}$, $O_{\rm SR}(G_2)=\mathcal{N}$ and $O_{\rm SR}(G_3)=\mathcal{B}$.~\hfill
\end{proof}

Next, we examine the relationship between $O_{\rm SR}(G)$ and $O_{\rm R}(G)$. We begin with the following useful observation.

\begin{observation}\emph{\cite{sebo}}\label{obs_r,sr}
For any connected graph $G$, a strong resolving set of $G$ is a resolving set of $G$.
\end{observation}

To facilitate statements, let us identify $\mathcal{B}, \mathcal{N}, \mathcal{M}$ with $-1, 0, 1$ respectively. Observation~\ref{obs_r,sr}, along with the rules of MBRG and MBSRG, implies $O_{\rm SR}(G)\leq O_{\rm R}(G)$. We thus have the following result.

\begin{corollary}\label{thm_comparison}
The set $\{(\mathcal{B}, \mathcal{B}), (\mathcal{B}, \mathcal{N}), (\mathcal{B}, \mathcal{M}), (\mathcal{N}, \mathcal{N}), (\mathcal{N}, \mathcal{M}), (\mathcal{M}, \mathcal{M})\}$ is a codomain of the ordered pair $(O_{\rm SR}(G), O_{\rm R}(G))$, viewed as a function on the space of graphs.
\end{corollary}

We next show that the codomain identified in Corollary~\ref{thm_comparison} is in fact the range of $(O_{\rm SR}(G), O_{\rm R}(G))$.

\begin{proposition}\label{examples}\hfill{}
\begin{itemize}
\item[(a)] There exist graphs $G$ satisfying $O_{\rm SR}(G)=\mathcal{M}=O_{\rm R}(G)$.
\item[(b)] There exist graphs $G$ satisfying $O_{\rm SR}(G)=\mathcal{N}=O_{\rm R}(G)$.
\item[(c)] There exist graphs $G$ satisfying $O_{\rm SR}(G)=\mathcal{B}=O_{\rm R}(G)$.
\item[(d)] There exist graphs $G$ satisfying $O_{\rm SR}(G)=\mathcal{N}$ and $O_{\rm R}(G)=\mathcal{M}$.
\item[(e)] There exist graphs $G$ satisfying $O_{\rm SR}(G)=\mathcal{B}$ and $O_{\rm R}(G)=\mathcal{M}$.
\item[(f)] There exist graphs $G$ satisfying $O_{\rm SR}(G)=\mathcal{B}$ and $O_{\rm R}(G)=\mathcal{N}$.
\end{itemize}
\end{proposition}

\begin{proof}
(a) Let $G$ be an $n$-path given by $u_1, u_2, \ldots, u_n$, where $n\ge 2$. Then $G_{\rm SR}= K_2$ and $\{\{u_1, u_n\}\}$ is a pairing (strong) resolving set of $G$. Thus, $O_{\rm SR}(G)=\mathcal{M}=O_{\rm R}(G)$ by Corollary~\ref{outcome_sr}(a) and Proposition~\ref{outcome_r}(a).

(b) Let $G\in\{K_3, K_{1,3}\}$. Then $G_{\rm SR}= C_3$ and $G$ has exactly one twin equivalence class of cardinality $3$ with $\dim(G)=2$. So, $O_{\rm SR}(G)=\mathcal{N}=O_{\rm R}(G)$ by Corollary~\ref{outcome_sr}(b) and Proposition~\ref{outcome_r}(c).

(c) Let $G$ be the star $K_{1,x}$ on ($x+1$) vertices, where $x\ge4$. Then $G_{\rm SR} = K_x$ and the leaves of $G$ form a twin equivalence class of cardinality $x$, where $x\ge4$. So, $O_{\rm SR}(G)=\mathcal{B}=O_{\rm R}(G)$ by Corollary~\ref{outcome_sr}(c) and Proposition~\ref{outcome_r}(d).

(d) Let $G$ be a tree with $ex(G)=1$ and $\sigma(G)=3$ such that $\ell_1, \ell_2, \ell_3$ are the terminal vertices of the exterior major vertex $v$ in $G$ with $d(v, \ell_1)\ge 2$, $d(v, \ell_2)\ge 2$ and $d(v,\ell_3)\ge1$. For each $i\in[2]$, let $s_i$ be the support vertex lying on the $v-\ell_i$ path. Then $G_{\rm SR}= C_3$ and $\{\{s_1, \ell_1\}, \{s_2, \ell_2\}\}$ is a pairing resolving set of $G$. So, $O_{\rm SR}(G)=\mathcal{N}$ by Corollary~\ref{outcome_sr}(b) and $O_{\rm R}(G)=\mathcal{M}$ by  Proposition~\ref{outcome_r}(a).

(e) Let $G$ be a tree with $ex(G)=1$ and $\sigma(G)=x\ge4$. Let $v$ be the exterior major vertex and $\ell_1, \ell_2, \ldots, \ell_x$ be the terminal vertices of $v$ in $G$ such that $d(v, \ell_x)\ge1$ and $d(v, \ell_i)\ge 2$ for each $i\in[x-1]$. For each $i\in[x-1]$, let $s_i$ be the support vertex lying on the $v-\ell_i$ path. Since $G_{\rm SR}= K_x \supseteq K_4$, $O_{\rm SR}(G)=\mathcal{B}$ by Corollary~\ref{outcome_sr}(c). Since $\cup_{i=1}^{x-1}\{\{s_i, \ell_i\}\}$ is a pairing resolving set of $G$, $O_{\rm R}(G)=\mathcal{M}$ by Proposition~\ref{outcome_r}(a).

(f) Let $G$ be a tree with $ex(G)=1$ and $\sigma(G)=x\ge4$. Let $v$ be the exterior major vertex and $\ell_1, \ell_2, \ldots, \ell_x$ be the terminal vertices of $v$ in $G$ such that $d(v, \ell_{x-2})=d(v, \ell_{x-1})=d(v, \ell_x)=1$ and $d(v, \ell_i)\ge 2$ for each $i\in[x-3]$. For each $i\in[x-3]$, let $s_i$ be the support vertex lying on the $v-\ell_i$ path in $G$. Since $G_{\rm SR}=K_x \supseteq K_4$, $O_{\rm SR}(G)=\mathcal{B}$ by Corollary~\ref{outcome_sr}(c). Next, we show that $O_{\rm R}(G)=\mathcal{N}$. Note that, for any resolving set $W$ of $G$, $|W \cap \{\ell_{x-2},\ell_{x-1}, \ell_{x}\}|\ge 2$ by Observation~\ref{obs_twin}(a). Let $X=(\cup_{i=1}^{x-3}\{\{s_i, \ell_i\}\}) \cup \{\{\ell_{x-2}, \ell_{x-1}\}\}$, and let $Z \subseteq V(G)$ with $|Z|=x-2$ such that $Z \cap \{\ell_{x-2}, \ell_{x-1}\} \neq\emptyset$ and $Z \cap \{s_i, \ell_i\} \neq\emptyset$ for each $i\in[x-3]$. Since $\{\ell_x\} \cup Z$ forms a minimum resolving set of $G$, $X$ is a quasi-pairing resolving set of $G$. So, $O_{\rm R}(G) \in \{\mathcal{M}, \mathcal{N}\}$ by Proposition~\ref{outcome_r}(b). In the B-game, Breaker can select two vertices of $\{\ell_{x-2},\ell_{x-1}, \ell_{x}\}$, and thus preventing Maker from occupying vertices that form a resolving set of $G$. Thus, $O_{\rm R}(G)=\mathcal{N}$.~\hfill
\end{proof}


\section{Some graph classes}\label{sec_graphs}

In this section, we examine the outcome of MBSRG played on some graph classes such as trees, cycles, the Petersen graph, complete multipartite graphs, corona product graphs, Cartesian product graphs, and modular product graphs.


\subsection{Trees, Cycles, Petersen Graph, Complete Multipartite Graphs}

We begin with the following observation for MBSRG that is analogous to the result obtained in~\cite{mbrg} for MBRG. We provide its proof to be self-contained.

\begin{observation}\label{obs_Nvertices}
Let $G$ be a connected graph of order $n\ge2$.
\begin{itemize}
\item[(a)] If $O_{\rm SR}(G)=\mathcal{M}$, then $\sdim(G) \le \lfloor\frac{n}{2}\rfloor$.
\item[(b)] If $\sdim(G)\ge \lceil\frac{n}{2}\rceil+1$, then $O_{\rm SR}(G)=\mathcal{B}$.
\end{itemize}
\end{observation}

\begin{proof}
Let $G$ be a connected graph of order $n\ge2$.

(a) Suppose $O_{\rm SR}(G)=\mathcal{M}$. In the B-game, Maker can select at most $\lfloor\frac{n}{2}\rfloor$ vertices in the course of MBSRG. Since Maker must occupy at least $\sdim(G)$ vertices, it follows that $\sdim(G) \le \lfloor\frac{n}{2}\rfloor$.

(b) Suppose $\sdim(G)\ge \lceil\frac{n}{2}\rceil+1$. Since Maker can occupy at most $\lceil\frac{n}{2}\rceil$ vertices in the course of MBSRG, Maker fails to occupy vertices that form a strong resolving set of $G$. Thus, $O_{\rm SR}(G)=\mathcal{B}$.~\hfill
\end{proof}

Next, we recall the strong metric dimension of trees, cycles, the Petersen graph and complete multi-partite graphs.

\begin{proposition}\label{sdim_graph}\hfill{}
\begin{itemize}
\item[(a)] \emph{\cite{sebo}} For any tree $T$ of order at least two, $\sdim(T)=\sigma(T)-1$.
\item[(b)] \emph{\cite{sebo}} For $n \ge 3$, $\sdim(C_n)=\lceil\frac{n}{2}\rceil$.
\item[(c)] \emph{\cite{yi_petersen}} For the Petersen graph $\mathcal{P}$, $\sdim(\mathcal{P})=8$.
\item[(d)] \emph{\cite{yi_petersen}} For $k \ge 2$, let $G=K_{a_1,\ldots, a_k}$ be a complete $k$-partite graph of order $n=\sum_{i=1}^{k}a_i \ge 3$. Let $s$ be the number of partite sets of $G$ consisting of exactly one element. Then
\begin{equation*}
\sdim(G)=\left\{
\begin{array}{ll}
n-k & \mbox{ if }s=0,\\
n-k+s-1 & \mbox{ if }s\neq 0.
\end{array}\right.
\end{equation*}
\end{itemize}
\end{proposition}

Next, we determine $O_{\rm SR}(G)$ when $G$ is a tree, a cycle, the Petersen graph and a complete multipartite graph, respectively.

\begin{proposition}\hfill{}\label{mix_graph_outcome}
\begin{itemize}
\item[(a)] For any tree $T$ of order at least two,
\begin{equation*}
O_{\rm SR}(T)=\left\{
\begin{array}{ll}
\mathcal{M} & \mbox{ if $\sigma(T)=2$},\\
\mathcal{N} & \mbox{ if $\sigma(T)=3$},\\
\mathcal{B} & \mbox{ if $\sigma(T)\ge4$}.
\end{array}
\right.
\end{equation*}

\item[(b)] For $n\ge 3$,
\begin{equation*}
O_{\rm SR}(C_n)=\left\{
\begin{array}{ll}
\mathcal{M} & \mbox{ if $n$ is even},\\
\mathcal{N} & \mbox{ if $n=3$},\\
\mathcal{B} & \mbox{ if $n\ge 5$ and $n$ is odd}.
\end{array}
\right.
\end{equation*}

\item[(c)] For the Petersen graph $\mathcal{P}$, $O_{\rm SR}(\mathcal{P})=\mathcal{B}$.

\item[(d)] For $m \ge 2$, let $G=K_{a_1,\ldots, a_m}$ be a complete $m$-partite graph of order $n=\sum_{i=1}^{m}a_i \ge 3$. Let $s$ be the number of partite sets of $G$ consisting of exactly one element. Then
\begin{equation*}
O_{\rm SR}(G)=\left\{
\begin{array}{ll}
\mathcal{B} & \mbox{ if } s\ge4 \mbox{ or } a_i\ge 4 \mbox{ for some } i\in[m],\\
{}                 & \mbox{ if } s=a_i=3 \mbox{ for some }i\in[m],\\
{}                 & \mbox{ if } a_i=a_j=3 \mbox{ for distinct } i,j\in[m],\\
\mathcal{N} & \mbox{ if } s=3 \mbox{ and } a_i\le 2 \mbox{ for each } i\in[m],\\
{}                 & \mbox{ if } s\le 2, a_i= 3 \mbox{ for exactly one }i\in[m], \mbox{ and } a_j\le2 \mbox{ for each }j\in[m]-\{i\},\\
\mathcal{M} & \mbox{ if } \max\{s, a_i\}\le 2 \mbox{ for each }i\in[m].
\end{array}
\right.
\end{equation*}
\end{itemize}
\end{proposition}

\begin{proof}
(a) Let $T$ be a tree of order at least two with $\sigma(T)=\sigma$; then $T_{\rm SR}=K_{\sigma}$ with $\sigma\ge2$. If $\sigma=2$, then $T_{\rm SR}=K_2$ and $O_{\rm SR}(T)=\mathcal{M}$ by Corollary~\ref{outcome_sr}(a). If $\sigma=3$, then $T_{\rm SR} = K_3$ and $O_{\rm SR}(T)=\mathcal{N}$ by Corollary~\ref{outcome_sr}(b). If $\sigma\ge 4$, then $T_{\rm SR} =K_{\sigma}\supseteq K_4$; thus, $O_{\rm SR}(T)=\mathcal{B}$ by Corollary~\ref{outcome_sr}(c).

(b) Let $n\ge 3$. First, suppose $n$ is even; then $\sdim(C_n)=\frac{n}{2}$ by Proposition~\ref{sdim_graph}(b). Since $(C_n)_{\rm SR} \cong \frac{n}{2}K_2$, we have $O_{\rm SR}(C_n)=\mathcal{M}$ by Corollary~\ref{outcome_sr}(a).

Second, suppose $n$ is odd; then $\sdim(C_n)=\frac{n+1}{2}$ by Proposition~\ref{sdim_graph}(b) and $(C_n)_{\rm SR} = C_n$. If $n=3$, then $O_{\rm SR}(C_3)=\mathcal{N}$ by Corollary~\ref{outcome_sr}(b). If $n\ge5$, then $O_{\rm SR}(C_n)=\mathcal{B}$ by Corollary~\ref{outcome_sr}(c).

(c) For the Petersen graph  $\mathcal{P}$, we have $\sdim(\mathcal{P})=8>6=\lceil\frac{|V(\mathcal{P})|}{2}\rceil+1$ by Proposition~\ref{sdim_graph}(c). By Observation~\ref{obs_Nvertices}(b), $O_{\rm SR}(\mathcal{P})=\mathcal{B}$.

(d) Let $G$ be a complete $m$-partite graph described in the statement of this proposition. First, suppose $s\ge4$ or $a_i\ge 4$ for some $i\in [m]$. Then $G_{\rm SR} \supseteq K_4$ and $O_{\rm SR}(G)=\mathcal{B}$ by Corollary~\ref{outcome_sr}(c).

Second, suppose $\max\{s, a_i\}= 3$ for each $i\in[m]$. If $s=a_i=3$ for some $i\in[m]$ or $a_i=a_j=3$ for distinct $i,j\in[m]$, then $G_{\rm SR} \supseteq 2C_3$ and $O_{\rm SR}(G)=\mathcal{B}$ by Corollary~\ref{outcome_sr}(c). If $s=3$ and $a_i \le2$ for each $i\in[m]$, then $G_{\rm SR} = C_3 \cup (m-3)K_2$ and $O_{\rm SR}(G)=\mathcal{N}$ by Corollary~\ref{outcome_sr}(b). If $s\in\{1,2\}$ and $a_i=3$ for exactly one $i\in[m]$, then $G_{\rm SR} = C_3 \cup (m-2)K_2$ and $O_{\rm SR}(G)=\mathcal{N}$ by Corollary~\ref{outcome_sr}(b). If $s=0$ and $a_i=3$ for exactly one $i\in[m]$, then $G_{\rm SR} = C_3 \cup (m-1)K_2$ and $O_{\rm SR}(G)=\mathcal{N}$ by Corollary~\ref{outcome_sr}(b).

Third, suppose $\max\{s, a_i\} \le 2$ for each $i\in[m]$. If $s=0$, then $G_{\rm SR}=m K_2$. If $s\in\{1,2\}$, then $G_{\rm SR}= (m-1)K_2$. In each case, $O_{\rm SR}(G)=\mathcal{M}$ by Corollary~\ref{outcome_sr}(a).~\hfill
\end{proof}


\subsection{Corona Product Graphs}

Let $G$ and $H$ be two graphs of order $n$ and $m$, respectively, and let $V(G)=\{u_1, u_2, \dots, u_n\}$. The \emph{corona product} $G \odot H$ is obtained from $G$ and $n$ copies of $H$, say $H_1, H_2, \ldots, H_n$, by adding an edge from $u_i$ to every vertex of $H_i$ for each $i\in[n]$. The \emph{join} of two graphs $G$ and $H$, denoted by $G+H$, is the graph obtained from the disjoint union of $G$ and $H$ by joining an edge between each vertex of $G$ and each vertex of $H$. We note that $K_1 \odot H=K_1+H$. For the strong metric dimension of corona product graphs, see~\cite{sdim_corona}; for the fractional strong metric dimension of corona product graphs, see~\cite{sdimF_products}. We first consider $O_{\rm SR}(G \odot H)$ when $G$ is a connected graph of order at least two.

\begin{proposition}
Let $G$ be a connected graph of order $n\ge 2$, and let $H$ be a graph of order $m\ge1$.
\begin{itemize}
\item[(a)] If $n\ge4$, then $O_{\rm SR} (G\odot H)=\mathcal{B}$.
\item[(b)] If $n=3$, then $O_{\rm SR}(G\odot H)=\left\{
\begin{array}{ll}
\mathcal{N} & \mbox{ if }m=1,\\
\mathcal{B} & \mbox{ if }m\ge2.
\end{array}
\right.$
\item[(c)] If $n=2$, then $O_{\rm SR}(G\odot H)=\left\{
\begin{array}{ll}
\mathcal{M} & \mbox{ if }m=1,\\
\mathcal{B} & \mbox{ if }m\ge2.
\end{array}
\right.$
\end{itemize}
\end{proposition}

\begin{proof}
Let $G$ be a connected graph of order $n\ge2$ and $H$ be a graph of order $m\ge1$. Let $V(G)=\{u_1,u_2, \ldots, u_n\}$ and let $H_1, \ldots, H_n$ be copies of $H$. For each $i\in[n]$, let $V(H_i)=\{w_{i,1},  \ldots, w_{i,m} \}$, where the vertex $w_{s,x}$ in $H_s$ corresponds to the vertex $w_{t,x}$ in $H_t$ for distinct $s,t\in[n]$ and $x\in[m]$. Since $n\ge2$, we have $V(G) \cap V((G\odot H)_{\rm SR})=\emptyset$.

(a) Let $n\ge4$. For distinct $i,j\in[n]$, we have $w_{i,1}$ MMD $w_{j,1}$ in $G \odot H$. So, $(G \odot H)_{\rm SR} \supseteq K_4$ and $O_{\rm SR} (G\odot H)=\mathcal{B}$ by Corollary~\ref{outcome_sr}(c).

(b) Let $n=3$. If $m=1$, then $(G \odot H)_{\rm SR} = C_3$ and $O_{\rm SR} (G\odot H)=\mathcal{N}$ by Corollary~\ref{outcome_sr}(b). If $m\ge2$, we have $w_{i,1}$ MMD $w_{j,1}$ and $w_{i,2}$ MMD $w_{j,2}$ in $G \odot H$ for any distinct $i,j\in[3]$; thus, $(G \odot H)_{\rm SR} \supseteq 2C_3$ and $O_{\rm SR} (G\odot H)=\mathcal{B}$ by Corollary~\ref{outcome_sr}(c).

(c) Let $n=2$. If $m=1$, then $(G \odot H)_{\rm SR} = K_2$ and $O_{\rm SR}(G\odot H)=\mathcal{M}$ by Corollary~\ref{outcome_sr}(a). So, suppose $m\ge2$. Since $\diam(G \odot H)=3=d_{G \odot H}(w_{i,x}, w_{j,z})$ for $\{i,j\}=\{1,2\}$ and $x, z\in[m]$, we have $(G\odot H)_{\rm SR} \supseteq C_4$. By Corollary~\ref{outcome_sr}(c), $O_{\rm SR}(G\odot H)=\mathcal{B}$.~\hfill
\end{proof}

Next, we consider $O_{\rm SR}(K_1 \odot H)$ when $H$ is a disconnected graph. We obtain the following lemma that is used in proving Proposition~\ref{corona_disconnected}. 

\begin{lemma}\label{outcome_sr_star}
Let $G$ be a connected graph. If $G_{\rm SR} = K_1+(aK_2\cup bK_1)$ such that $a\ge1$ and $b\ge0$, or $a=0$ and $b\ge2$, then $O_{\rm SR}(G)=\mathcal{N}$.
\end{lemma}

\begin{proof}
Let $G_{\rm SR} = K_1+(aK_2\cup bK_1)$, where $a\ge1$ and $b\ge0$, or $a=0$ and $b\ge2$. If $G_{\rm SR}=K_1+K_2=C_3$ (when $a=1$ and $b=0$) or $G_{\rm SR}=K_1+2K_1=P_3$ (when $a=0$ and $b=2$), then $O_{\rm SR}(G)=\mathcal{N}$ by Corollary~\ref{outcome_sr}(b). So, suppose $G_{\rm SR} \not\in\{C_3, P_3\}$. Since $\Delta(G_{\SR})\ge 2$, $O_{\rm SR}(G)\in\{\mathcal{N}, \mathcal{B}\}$ by Theorem~\ref{main_result}(c). To prove $O_{\rm SR}(G)=\mathcal{N}$, we need to show that Maker wins the M-game of MBSRG. Let $v$ be the vertex in the $K_1 \subset G_{\rm SR}$ with $\deg(v)=2a+b$. If $a\neq0$, let $\{u_1, u_2, \ldots, u_{2a}\}$ be the set of degree-$2$ vertices in $G_{\rm SR}$ such that $u_{2i-1}u_{2i}\in E(G_{\rm SR})$ for each $i\in[a]$; let $X=\cup_{i=1}^{a}\{\{u_{2i-1},u_{2i}\}\}$. We consider the M-game. If $a=0$, then Maker can choose $v$ on his first move. If $a\neq 0$, then Maker can select $v$ on his first move and exactly one vertex of each pair in $X$ afterwards. In each case, Maker is able to occupy a vertex cover of $G_{\rm SR}$, and thus Maker wins the M-game of MBSRG.~\hfill
\end{proof}

\begin{proposition}\label{corona_disconnected}
Let $c(H)$ denote the number of connected components of a graph $H$ with $c(H)\ge2$.
\begin{itemize}
\item[(a)] If $c(H)\ge4$, then $O_{\rm SR} (K_1\odot H)=\mathcal{B}$.
\item[(b)] If $c(H)=3$, then $O_{\rm SR}(K_1\odot H)=\left\{
\begin{array}{ll}
\mathcal{N} & \mbox{ if }H = 3K_1,\\
\mathcal{B} & \mbox{ otherwise}.
\end{array}
\right.$
\item[(c)] If $c(H)=2$, then $O_{\rm SR}(K_1\odot H)=\left\{
\begin{array}{ll}
\mathcal{M} & \mbox{ if }H = 2K_1,\\
\mathcal{N} & \mbox{ if }H \in \{K_1 \cup K_2,K_1 \cup P_3\} \mbox{ or }H= K_1 \cup H^*\\
\mathcal{B} & \mbox{ otherwise},
\end{array}
\right.$
\end{itemize}
where $H^*$ is a connected graph of order $m\ge4$ such that $\deg_{H^*}(w) \in \{m-2, m-1\}$ for each $w\in V(H^*)$ with at most two vertices in $H^*$ having degree $m-1$.
\end{proposition}

\begin{proof}
Let $H$ be a disconnected graph. For $c=c(H)\ge2$, let $H^1, H^2,\ldots, H^c$ be the connected components of $H$ of order $m_1, m_2, \ldots, m_c$, respectively; let $V(H^i)=\{u_{i,1}, \ldots, u_{i,m_i}\}$ for each $i\in[c]$. Since $H$ is disconnected, $\diam(K_1\odot H) = 2$.

(a) Let $c(H)\ge4$. Since $u_{i,1}$ MMD $u_{j,1}$ in $K_1\odot H$ for any distinct $i,j\in[c]$, we have $(K_1\odot H)_{\rm SR} \supseteq K_4$. By Corollary~\ref{outcome_sr}(c), $O_{\rm SR} (K_1\odot H)=\mathcal{B}$.

(b) Let $c(H)=3$. If $m_i=1$ for each $i\in[3]$, then $H = 3K_1$ and $(K_1 \odot H)_{\rm SR}= C_3$; thus, $O_{\rm SR} (K_1\odot H)=\mathcal{N}$ by Corollary~\ref{outcome_sr}(b). If $m_i\ge2$ for some $i\in [3]$, say $m_3\ge2$ without loss of generality, then $(K_1 \odot H)_{\rm SR} \supseteq C_4$ since $u_{3,1}$ MMD $u_{i,1}$ and $u_{3,2}$ MMD $u_{i,1}$ in $K_1\odot H$ for each $i\in[2]$; thus, $O_{\rm SR} (K_1\odot H)=\mathcal{B}$ by Corollary~\ref{outcome_sr}(c).

(c) Let $c(H)=2$. First, suppose $m_i=1$ for each $i\in[2]$. Then $H = 2K_1$ and $(K_1 \odot H)_{\rm SR} = K_2$; thus, $O_{\rm SR} (K_1\odot H)=\mathcal{M}$ by Corollary~\ref{outcome_sr}(a).

Second, suppose $m_i=1$ for exactly one $i\in[2]$; let $m_1=1$ and $m_2\ge2$ without loss of generality. If $\diam(H^2)=1$ and $m_2=2$, then $H = K_1 \cup K_2$ and $(K_1 \odot H)_{\rm SR} = C_3$; thus, $O_{\rm SR} (K_1\odot H)=\mathcal{N}$ by Corollary~\ref{outcome_sr}(b). If $\diam(H^2)=1$ and $m_2\ge3$, then $H = K_1 \cup K_{m_2}$ and $(K_1 \odot H)_{\rm SR} =K_{m_2+1} \supseteq K_4$; thus, $O_{\rm SR} (K_1\odot H)=\mathcal{B}$ by Corollary~\ref{outcome_sr}(c).

Now, let $\diam(H^2)=2$; then $m_2\ge3$. If $m_2=3$, then $H = K_1 \cup P_3$ and $(K_1 \odot H)_{\rm SR} =K_1+(K_2 \cup K_1)$; thus, $O_{\rm SR} (K_1\odot H)=\mathcal{N}$ by Lemma~\ref{outcome_sr_star}. So, suppose $m_2\ge4$. If $H^2$ has at least three vertices of degree $m_2-1$, say $\deg_{H^2}(u_{2,i})=m_2-1$ for each $i\in[3]$, then any two distinct vertices in $\{u_{1,1}, u_{2,1}, u_{2,2}, u_{2,3}\}$ are MMD pairs in $K_1 \odot H$; thus, $(K_1 \odot H)_{\rm SR} \supseteq K_4$ and $O_{\rm SR} (K_1\odot H)=\mathcal{B}$ by Corollary~\ref{outcome_sr}(c). If $H^2$ has a vertex of degree at most $m_2-3$, say $\deg_{H^2}(u_{2,1}) \le m_2-3$ with $u_{2,1}u_{2, m_2-1} \not\in E(H^2)$ and $u_{2,1}u_{2, m_2} \not\in E(H^2)$, then $u_{1,1}$ MMD $u_{2,j}$ and $u_{2,1}$ MMD $u_{2,j}$ in $K_1\odot H$ for each $j\in\{m_2-1, m_2\}$; thus, $(K_1\odot H)_{\rm SR} \supseteq C_4$ and $O_{\rm SR} (K_1\odot H)=\mathcal{B}$ by Corollary~\ref{outcome_sr}(c). If the degree of each vertex in $H^2$ is either $m_2-2$ or $m_2-1$ with at most two vertices of $H^2$ having degree $m_2-1$, then $(K_1 \odot H)_{\rm SR} =K_1+\frac{m_2}{2}K_2$ (when $m_2$ is even) or $(K_1 \odot H)_{\rm SR} = K_1+(\frac{m_2-1}{2}K_2 \cup K_1)$ (when $m_2$ is odd); thus, $O_{\rm SR} (K_1\odot H)=\mathcal{N}$ by Lemma~\ref{outcome_sr_star}.

If $\diam(H^2)=d \ge 3$, say $w_0, w_1, w_2, \ldots, w_d$ is a diametral path in $H^2$, then $u_{1,1}$ MMD $w_i$ and $w_0$ MMD $w_i$ in $K_1\odot H$ for each $i\in\{d-1,d\}$; thus, $(K_1\odot H)_{\rm SR} \supseteq C_4$ and $O_{\rm SR} (K_1\odot H)=\mathcal{B}$ by Corollary~\ref{outcome_sr}(c).

Third, suppose $m_i\ge2$ for each $i\in[2]$. In this case, $u_{1,i}$ MMD $u_{2,j}$ in $K_1\odot H$ for each $i\in[m_1]$ and for each $j\in[m_2]$. So, $(K_1 \odot H)_{\rm SR} \supseteq C_4$ and $O_{\rm SR} (K_1\odot H)=\mathcal{B}$ by Corollary~\ref{outcome_sr}(c).~\hfill
\end{proof}

Next, we consider $O_{\rm SR}(K_1 \odot H)$ when $H$ is a connected graph. We begin by considering the outcome of MBSRG played on fan graphs and wheel graphs.

\begin{example}\label{cor_path}
For $n\ge1$, $O_{\rm SR}(K_1\odot P_n)=O_{\rm SR}(K_1+P_n)=\left\{
\begin{array}{ll}
\mathcal{M} & \mbox{ if }n\in\{1,3\},\\
\mathcal{N} & \mbox{ if }n\in\{2,4\},\\
\mathcal{B} & \mbox{ if }n\ge5.
\end{array}
\right.$
\end{example}

\begin{proof}
First, let $n\in\{1,2,3,4\}$. Since $(K_1\odot P_1)_{\rm SR}=K_2$ and $(K_1\odot P_3)_{\rm SR}=2K_2$, we have $O_{\rm SR}(K_1\odot P_1)=O_{\rm SR}(K_1\odot P_3)=\mathcal{M}$ by Corollary~\ref{outcome_sr}(a). Since $(K_1\odot P_2)_{\rm SR}=C_3$ and $(K_1\odot P_4)_{\rm SR}=P_4$, we have $O_{\rm SR}(K_1\odot P_2)=O_{\rm SR}(K_1\odot P_4)=\mathcal{N}$ by Corollary~\ref{outcome_sr}(b). Next, let $n\ge5$; suppose the $n$-path of $K_1\odot P_n$ be given by $u_1, u_2, \ldots, u_n$. Since $u_1$ MMD $u_j$ and $u_2$ MMD $u_j$ in $K_1\odot P_n$ for each $j\in\{n-1, n\}$, where $n\ge5$, $(K_1 \odot P_n)_{\rm SR} \supseteq C_4$ and $O_{\rm SR}(K_1\odot P_n)=\mathcal{B}$ by Corollary~\ref{outcome_sr}(c).~\hfill
\end{proof}

\begin{example}\label{cor_cycle}
For $n\ge3$, $O_{\rm SR}(K_1\odot C_n)=O_{\rm SR}(K_1+ C_n)=\left\{
\begin{array}{ll}
\mathcal{M} & \mbox{ if }n=4,\\
\mathcal{B} & \mbox{ otherwise}.
\end{array}
\right.$
\end{example}

\begin{proof}
First, let $n\in\{3,4,5\}$. Since $K_1\odot C_3 = K_4$, we have $(K_1\odot C_3)_{\rm SR} = K_4$ and $O_{\rm SR}(K_1\odot C_3)=\mathcal{B}$ by Corollary~\ref{outcome_sr}(c). Since $(K_1 \odot C_4)_{\rm SR} = 2 K_2$, we have $O_{\rm SR}(K_1\odot C_4)=\mathcal{M}$ by Corollary~\ref{outcome_sr}(a). Since $(K_1 \odot C_5)_{\rm SR} = C_5$, we have $O_{\rm SR}(K_1\odot C_5)=\mathcal{B}$ by Corollary~\ref{outcome_sr}(c).

Second, let $n\ge6$; notice that $\diam(K_1\odot C_n)=2$. Suppose the $n$-cycle of $K_1\odot C_n$ be given by $u_1, u_2, \ldots, u_n, u_1$. Since any pair of distinct vertices in $\{u_1, u_3, u_5\}$ and $\{u_2, u_4, u_6\}$, respectively, are MMD pairs in $K_1 \odot C_n$, we have $(K_1 \odot C_n)_{\rm SR} \supseteq 2C_3$. By Corollary~\ref{outcome_sr}(c), $O_{\rm SR}(K_1\odot C_n)=\mathcal{B}$ for $n\ge6$.~\hfill
\end{proof}

\begin{proposition}\label{corona_1connected}
Let $H$ be a connected graph of order $m\ge3$.
\begin{itemize}
\item[(a)] If $\diam(H)=1$, then $O_{\rm SR} (K_1\odot H) =\mathcal{B}$.
\item[(b)] If $\diam(H)=2$, then $O_{\rm SR} (K_1\odot H) \in\{\mathcal{M}, \mathcal{N}, \mathcal{B}\}$.
\item[(c)] If $\diam(H)=3$, then $O_{\rm SR} (K_1\odot H) \in\{\mathcal{N}, \mathcal{B}\}$.
\item[(d)] If $\diam(H)\ge4$, then $O_{\rm SR}(K_1\odot H)=\mathcal{B}$.
 \end{itemize}
\end{proposition}

\begin{proof}
Let $H$ be a connected graph of order $m\ge3$.

(a) Let $\diam(H)=1$. Then $H= K_m$ and $K_1\odot H = K_{m+1}$, where $m\ge3$. So, $(K_1 \odot H)_{\rm SR} = K_{m+1} \supseteq K_4$ and $O_{\rm SR} (K_1\odot H) =\mathcal{B}$ by Corollary~\ref{outcome_sr}(c).

(b) Let $\diam(H)=2$. We provide graphs $H$ with $\diam(H)=2$ such that $O_{\rm SR} (K_1\odot H)=\mathcal{M}$, $O_{\rm SR} (K_1\odot H)=\mathcal{N}$ and $O_{\rm SR} (K_1\odot H)=\mathcal{B}$, respectively. If $H\in\{P_3, C_4\}$, then $\diam(H)=2$ and $O_{\rm SR} (K_1\odot H)=\mathcal{M}$ (see Examples~\ref{cor_path} and~\ref{cor_cycle}). If $H\in\{K_{2,3}, K_{1,2,3}\}$, then $\diam(H)=2$, $K_1\odot K_{2,3}=K_{1,2,3}$ and $K_1 \odot K_{1,2,3}=K_{1,1,2,3}$; thus, $O_{\rm SR} (K_1\odot H)=\mathcal{N}$ by Proposition~\ref{mix_graph_outcome}(d). If $H=K_{a,b}$ with $\min\{a,b\}\ge3$, then $\diam(H)=2$ and $K_1\odot K_{a,b}=K_{1,a,b}$, where $a,b\ge3$; thus, $O_{\rm SR} (K_1\odot H)=\mathcal{B}$ by Proposition~\ref{mix_graph_outcome}(d).

(c) Let $\diam(H)=3$. Let $w_0, w_1, w_2, w_3$ be a diametral path in $H$. Since $w_0$ MMD $w_2$ and $w_0$ MMD $w_3$ in $K_1\odot H$, we have $\Delta((K_1 \odot H)_{\rm SR})\ge2$; thus, $O_{\rm SR}(K_1\odot H)\in\{\mathcal{N}, \mathcal{B}\}$ by Theorem~\ref{main_result}(c).

Next, we provide graphs $H$ with $\diam(H)=3$ such that $O_{\rm SR} (K_1\odot H)=\mathcal{N}$ and $O_{\rm SR} (K_1\odot H)=\mathcal{B}$, respectively. If $H=P_4$, then $\diam(H)=3$ and $O_{\rm SR} (K_1\odot H)=\mathcal{N}$ (see Example~\ref{cor_path}). If $H\in\{C_6, C_7\}$, then $\diam(H)=3$ and $O_{\rm SR} (K_1\odot H)=\mathcal{B}$ (see Example~\ref{cor_cycle}).

(d) Suppose $\diam(H) \ge4$. Let $w_0, w_1, \ldots, w_{d}$ be a diametral path in $H$, where $d\ge4$. Since $w_0$ MMD $w_j$ and $w_1$ MMD $w_j$ in $K_1\odot H$ for each $j\in\{d-1, d\}$, we have $(K_1 \odot H)_{\rm SR} \supseteq C_4$; thus, $O_{\rm SR}(K_1\odot H)=\mathcal{B}$ by Corollary~\ref{outcome_sr}(c).~\hfill
\end{proof}


\subsection{Cartesian Product Graphs}

The \emph{Cartesian product} of two graphs $G$ and $H$, denoted by $G \square H$, is the graph with the vertex set $V(G) \times V(H)$ such that $(u,w)$ is adjacent to $(u', w')$ if and only if either $u=u'$ and $ww'\in E(H)$,  or $w=w'$ and $uu'\in E(G)$. The \emph{direct product} (also known as the \emph{tensor product}) of two graphs $G$ and $H$, denoted by $G \times H$, is the graph with the vertex set $V(G) \times V(H)$ such that $(u,w)$ is adjacent to $(u', w')$ if and only if $uu'\in E(G)$ and $ww'\in E(H)$. For the strong metric dimension of Cartesian product graphs, see~\cite{GSR}; for the fractional strong metric dimension of Cartesian product graphs, see~\cite{sdimF_products}. We first recall the following result connecting Cartesian product and direct product.

\begin{theorem}\emph{\cite{GSR}}\label{cartesian_direct}
If $G$ and $H$ are two connected graphs of order at least two, then $(G \square H)_{\rm SR} = G_{\rm SR} \times H_{\rm SR}$.
\end{theorem}

Next, for connected graphs $G$ and $H$ of order at least two, we show that $O_{\rm SR}(G \square H)\in\{\mathcal{M},\mathcal{B}\}$ and we characterize Cartesian product graphs satisfying $O_{\rm SR}(G \square H)=\mathcal{M}$. We begin with the following lemma that is used in proving our main result on the outcome of MBSRG played on Cartesian product graphs.

\begin{lemma}\label{cartesian_P3}
Let $G$ and $H$ be two connected graphs of order at least two. If $\Delta(G_{\rm SR})\ge2$ or $\Delta(H_{\rm SR})\ge2$, then $O_{\rm SR}(G \square H)=\mathcal{B}$.
\end{lemma}

\begin{proof}
Without loss of generality, suppose $\Delta(G_{\rm SR}) \ge 2$. By Theorem~\ref{cartesian_direct}, $(G \square H)_{\rm SR} = G_{\rm SR} \times H_{\rm SR} \supseteq P_3 \times K_2 \supseteq 2P_3$. Thus, $O_{\rm SR} (G \square H)=\mathcal{B}$ by Corollary~\ref{outcome_sr}(c).~\hfill
\end{proof}

\begin{theorem}\label{theorem_cartesian}
Let $G$ and $H$ be two connected graphs of order at least two. Then $O_{\rm SR}(G \square H)\in\{\mathcal{M},\mathcal{B}\}$. Moreover,
$O_{\rm SR}(G \square H)=\mathcal{M}$ if and only if $G_{\rm SR} = aK_2$ and $H_{\rm SR} =bK_2$, where $a\in\mathbb{Z}^+$ and $b\in \mathbb{Z}^+$.
\end{theorem}

\begin{proof}
First, suppose $\Delta(G_{\rm SR})\ge2$ or $\Delta(H_{\rm SR})\ge2$. Then $O_{\rm SR} (G \square H)=\mathcal{B}$ by Lemma~\ref{cartesian_P3}. 

Second, suppose $\Delta(G_{\rm SR})=1=\Delta(H_{\rm SR})$; then $G_{\rm SR} = aK_2$ and $H_{\rm SR} = bK_2$, where $a,b\in \mathbb{Z}^+$. In this case, $(G \square H)_{\rm SR} = G_{\rm SR} \times H_{\rm SR} = aK_2 \times b K_2 = (2ab)K_2$, and thus $O_{\rm SR}(G \square H)=\mathcal{M}$ by Corollary~\ref{outcome_sr}(a).~\hfill
\end{proof}

We note that $(K_n)_{\rm SR} = K_n$ for $n\ge2$, $(C_{2m-1})_{\rm SR} = C_{2m-1}$ and $(C_{2m})_{\rm SR} = m K_{2}$ for $m\ge2$, and $(T_{\rm SR}) = K_{\sigma}$ for any tree $T$ with $\sigma$ leaves. So, Lemma~\ref{cartesian_P3} and Theorem~\ref{theorem_cartesian} imply the following.

\begin{corollary}
Let $G$ be a connected graph of order at least two.
\begin{itemize}
\item[(a)] For $n\ge3$, $O_{\rm SR}(G \square K_n)=\mathcal{B}$.
\item[(b)] For $n\ge3$, $O_{\rm SR}(G \square C_n)=\left\{
\begin{array}{ll}
\mathcal{M} & \mbox{ if $n$ is even and $\Delta(G_{\rm SR})=1$},\\ 
\mathcal{B} &  \mbox{ if $n$ is odd or $\Delta(G_{\rm SR})\ge2$}.
\end{array}\right.$
\item[(c)] For any tree $T$ that is not a path, $O_{\rm SR}(G \square T)=\mathcal{B}$.
\end{itemize}
\end{corollary}


\subsection{Modular Product Graphs}

The \emph{modular product} of two graphs $G$ and $H$, denoted by $G\diamond H$, is the graph with vertex set $V(G \diamond H)=V(G) \times V(H)$ and edge set $E(G\diamond H)=E(G\Box H)\cup E(G\times H)\cup E(\overline{G}\times\overline{H})$, where $\overline{X}$ denotes the complement of a graph $X$. The \emph{lexicographic product} of two graphs $G$ and $H$, denoted by $G \circ H$, is the graph with the vertex set $V(G) \times V(H)$ such that $(u,w)$ is adjacent to $(u', w')$ if and only if either $uu'\in E(G)$, or $u=u'$ and $ww'\in E(H)$. It was shown in~\cite{mod_product} that, if (i) neither $G$ nor $H$ is a complete graph and (ii) $G$ and $H$ are not both disjoint unions of two cliques, then $\diam(G \diamond H) \le 3$. 

First, we consider $O_{\rm SR}(G \diamond K_t)$ for $t\ge2$. It was observed in~\cite{mod_product} that $G\diamond K_t=G \circ K_t$.

\begin{proposition}
If $G$ is a connected graph of order at least two and $t\ge2$, then $O_{\rm SR}(G \diamond K_t)=\mathcal{B}$.
\end{proposition}

\begin{proof}
Suppose $d_G(u, w)=\diam(G)$. Then $u \ {\rm MMD} \ w$ in $G$. Let $^uK_t$ and $^wK_t$ be the copies of $K_t$ corresponding to the vertices $u$ and $w$ in $G \circ K_t$, respectively. Then, for $x\in V(^uK_t)$ and $y\in V(^wK_t)$, we have $x \ {\rm MMD} \ y$ in $G \circ K_t=G \diamond K_t$. So, $(G \diamond K_t)_{\rm SR} \supseteq K_{t,t}\supseteq C_4$ since $t\ge2$. Thus, $O_{\rm SR}(G \diamond K_t)=\mathcal{B}$ by Corollary~\ref{outcome_sr}(c).~\hfill
\end{proof}

Second, we consider $O_{\rm SR}(G \diamond H)$ when $\diam(G \diamond H)\le2$. Let ${\rm TW} (G)$ denote the set of all twin edges of $G$, where a twin edge is an edge between adjacent twins. We recall the following result on modular product graphs with diameter $2$.

\begin{theorem}\emph{\cite{mod_product}}\label{mod_prod_diam2}
If $\diam(G\diamond H)=2$, then
$$E((G\diamond H)_{\rm SR})={\rm TW}(G\diamond H)\cup E(\overline{G}\Box\overline{H})\cup E(G\times\overline{H})\cup E(\overline{G}\times H).$$
\end{theorem}

\begin{observation}
If $G$ and $H$ are graphs of order at least two with $\diam(G\diamond H)=1$, then $O_{\rm SR}(G \diamond H)=\mathcal{B}$.
\end{observation}

\begin{proof}
Suppose $G$ and $H$ are graphs of order $n$ and $m$, respectively, where $n, m\ge2$. Since $\diam(G\diamond H)=1$, $G \diamond H =K_{mn}$. So, $(G\diamond H)_{\rm SR}=K_{mn}\supseteq K_4$. By Corollary~\ref{outcome_sr}(c), $O_{\rm SR}(G \diamond H)=\mathcal{B}$.~\hfill
\end{proof} 

\begin{proposition}
Let $G$ and $H$ be connected graphs of order at least two such that neither $G$ nor $H$ is a complete graph. If $\diam(G\diamond H)=2$, then $O_{\rm SR}(G \diamond H)=\mathcal{B}$.
\end{proposition} 

\begin{proof}
Since neither $G$ nor $H$ is a complete graph and $\diam(G\diamond H)=2$, Theorem~\ref{mod_prod_diam2} implies $(G \diamond H)_{\rm SR} \supseteq \overline{G} \square \overline{H} \supseteq K_2 \square K_2 \supseteq C_4$. By Corollary~\ref{outcome_sr}(c), $O_{\rm SR}(G \diamond H)=\mathcal{B}$.~\hfill
\end{proof}
 
Next, we consider $O_{\rm SR}(G \diamond H)$ when $\diam(G \diamond H)=3$. We show that $O_{\rm SR}(G \diamond H) \in \{\mathcal{M}, \mathcal{N}, \mathcal{B} \}$. We recall some terminology and notations. For $x\in V(G)$, let $N_G[x]=N_G(x) \cup\{x\}$. A set $D \subseteq V(G)$ is a \emph{dominating set} of $G$ if $\cup_{x\in D}N[x]=V(G)$. If $\{u, w\}$ is a minimum dominating set of $G$ satisfying $N_G[u] \cap N_G[w]=\emptyset$ and $N_G[u] \cup N_G[w]=V(G)$, then the pair $\{u, w\}$ is called a $\gamma_G$-pair. Let $\mathbb{P}(G)=\{x\in V(G): \exists y\in V(G) \mbox{ such that }\{x,y\} \mbox{ is a $\gamma_G$-pair}\}$. The \emph{$\gamma_G$-pair graph} of $G$, denoted by $\mathbb{GP}(G)$, has vertex set $V(\mathbb{GP}(G))=V(G)$ and $uv\in E(\mathbb{GP}(G))$ if and only if $\{u, v\}$ is a $\gamma_G$-pair. A vertex $u\in V(G)$ is called a \emph{boundary vertex} of $G$ if there is a vertex $v\in V(G)$  such that $d_G(u,v)=\diam(G)$. We recall the following important results in modular product graphs.

\begin{theorem}\emph{\cite{mod_product}}\label{mod_prod_dist3}
Suppose neither $G$ nor $H$ is a complete graph, and that $G$ and $H$ are not both disjoint unions of two cliques.
Then $d_{G\diamond H}((g,h),(g',h'))=3$ if and only if at least one of the following two (symmetric) conditions holds
\begin{equation}\label{eq1}
N_G[g]=N_G[g'] \wedge d_H(h,h') \geq 3 \wedge  \left( N_G[g]=V(G) \vee \{ h,h' \} \text{ is a } \gamma_H\text{-pair}\right)
\end{equation}
\begin{equation}\label{eq2}
N_H[h]=N_H[h'] \wedge d_G(g,g') \geq 3 \wedge  \left( N_H[h]=V(H) \vee \{ g,g' \} \text{ is a } \gamma_G\text{-pair}\right)
\end{equation}
\end{theorem}

\begin{theorem}\emph{\cite{mod_product}}\label{twins1}
Let $G$ and $H$ be graphs of order at least two. 
\begin{itemize}
\item[(a)] There are no distinct non-adjacent twins in $G\diamond H$.
\item[(b)] Two vertices $(g,h)$ and $(g',h')$ are distinct adjacent twins in $G\diamond H$ if and only if
\begin{itemize}
\item[(i)] $N_G[g]=N_G[g']$ and $N_H[h]=N_H[h']$ with $(g,h)\neq(g',h')$, or
\vspace*{-0.01in}
\item[(ii)] $\{g,g'\}$ is a $\gamma_G$-pair and $\{h,h'\}$ is a $\gamma_H$-pair.
\end{itemize}
\end{itemize}
\end{theorem}

\begin{theorem}\emph{\cite{mod_product}}\label{SRedges}
For non-complete graphs $G$ and $H$ such that at least one of the two is not the disjoint union of two cliques, suppose $\diam(G\diamond H)=3$. Then, for vertices $(g,h)$ and $(g',h')$ of $G\diamond H$, $(g,h)(g',h')\in E((G\diamond H)_{\rm SR})$ if and only if the two vertices satisfy at least one of the following conditions:
\begin{itemize}
\item[$(a)$] $(g,h)$ and $(g',h')$ are distinct adjacent twins of $G\diamond H$;
\item[$(b)$] $d_{G\diamond H}((g,h),(g',h'))=2$, and both $(g,h)$ and $(g',h')$ are not boundary vertices in $G\diamond H$;
\item[$(c)$] $d_{G\diamond H}((g,h),(g',h'))=3$;
\item[$(d)$] $g$ and $g'$ are universal vertices in $G$, $d_H(h,h')=2$ and $hh'\in E(H_{\rm SR})$; or the mirror of the preceding condition;
\item[$(e)$] $g$ is a universal vertex but $g'$ is not a universal vertex in $G$, also $d_H(h,h')=2$ and $d_H(h,h'')\leq 2$ for every $h''\in N_H[h']$, also $h'$ does not belong to any $\gamma_H$-pair; or the mirror of the preceding condition.
\end{itemize}
\end{theorem}

Next, we consider MBSRG played on modular product graphs with $O_{\rm SR} (G \diamond H)=\mathcal{B}$.

\begin{theorem}\label{diamGmodH3_theorem1}
Let $G$ and $H$ be connected graphs of order at least two such that neither $G$ nor $H$ is a complete graph. Suppose $\diam(G\diamond H)=3$.
\begin{itemize}
\item[(a)] If $\{g, g'\}$ is a $\gamma_G$-pair and $\{h, h'\}$ is a $\gamma_H$-pair, then $O_{\rm SR}(G \diamond H)=\mathcal{B}$.

\item[(b)] If $G$ has at least two distinct universal vertices and $\diam(H)\ge3$, then $O_{\rm SR}(G \diamond H)=\mathcal{B}$.

\item[(c)] Suppose there's no $\gamma_G$-pair and no universal vertex in $G$, but $G$ has a pair of distinct adjacent twins and $H$ has a $\gamma_H$-pair. Then $O_{\rm SR}(G \diamond H)=\mathcal{B}$.

\item[(d)] Suppose there's no $\gamma_G$-pair and no universal vertex in $G$, and $H$ has a $\gamma_H$-pair and at least two distinct vertices in $H$ belong to no $\gamma_H$-pair. Then $O_{\rm SR}(G \diamond H)=\mathcal{B}$.

\item[(e)] Suppose there's no $\gamma_G$-pair and no universal vertex in $G$, and $H$ has distinct adjacent twins $h_1$ and $h_2$ such that both $\{h_1, h\}$ and $\{h_2, h\}$ are $\gamma_H$-pairs for some $h\in V(H)-\{h_1, h_2\}$. Then $O_{\rm SR}(G \diamond H)=\mathcal{B}$.

\end{itemize}
\end{theorem}

\begin{proof}
Let $G$ and $H$ be graphs described in the statement with $\diam(G\diamond H)=3$.

(a) Let $\{g, g'\}$ be a $\gamma_G$-pair and $\{h, h'\}$ be a $\gamma_H$-pair.  By Theorem~\ref{twins1}(b), $(g,h)$ and $(g', h')$ are distinct adjacent twins in $G \diamond H$, and $(g,h')$ and $(g', h)$ are distinct adjacent twins in $G \diamond H$; thus, $(g,h)(g', h')\in E((G \diamond H)_{\rm SR})$ and $(g,h')(g', h)\in E((G \diamond H)_{\rm SR})$ by Theorem~\ref{SRedges}. Also, by (\ref{eq1}) of Theorem~\ref{mod_prod_dist3}, we have $d_{G\diamond H}((g,h),(g,h'))=3$ and $d_{G\diamond H}((g',h),(g',h'))=3$; thus, $(g,h) (g,h') \in E((G \diamond H)_{\rm SR})$ and $(g',h) (g',h') \in E((G\diamond H)_{\rm SR})$ by Theorem~\ref{SRedges}. Since $(G \diamond H)_{\rm SR} \supseteq C_4$, $O_{\rm SR}(G \diamond H)=\mathcal{B}$ by Corollary~\ref{outcome_sr}(c).

(b) Let $g$ and $g'$ be distinct universal vertices in $G$ and $d_H(h,h') \ge3$. From (\ref{eq1}) of Theorem~\ref{mod_prod_dist3}, we have $d_{G\diamond H}((g,h),(g,h'))=3$, $d_{G\diamond H}((g',h),(g',h'))=3$, $d_{G\diamond H}((g,h),(g',h'))=3$ and $d_{G\diamond H}((g',h),(g,h'))=3$. So, $(g,h)(g,h')\in E((G\diamond H)_{\rm SR})$, $(g',h)(g',h')\in E((G\diamond H)_{\rm SR})$, $(g,h)(g',h')\in E((G\diamond H)_{\rm SR})$ and $(g',h)(g,h')\in E((G\diamond H)_{\rm SR})$ by Theorem~\ref{SRedges}. Since $(G \diamond H)_{\rm SR} \supseteq C_4$, $O_{\rm SR}(G \diamond H)=\mathcal{B}$ by Corollary~\ref{outcome_sr}(c).~ \hfill

(c) Let $g_1$ and $g_2$ be distinct adjacent twins in $G$; then $N_G[g_1]=N_G[g_2]$. Suppose $\{h_1, h_2\}$ is a $\gamma_H$-pair. From (\ref{eq1}) of Theorem~\ref{mod_prod_dist3}, we have $d_{G\diamond H}((g_1,h_1),(g_1,h_2))=3$, $d_{G\diamond H}((g_2,h_1),(g_2,h_2))=3$, $d_{G\diamond H}((g_1,h_1),(g_2,h_2))=3$ and $d_{G\diamond H}((g_2,h_1),(g_1,h_2))=3$. By Theorem~\ref{SRedges}, $(g_1,h_1)(g_1,h_2)\in E((G\diamond H)_{\rm SR})$, $(g_2,h_1)(g_2,h_2)\in E((G\diamond H)_{\rm SR})$, $(g_1,h_1)(g_2,h_2)\in E((G\diamond H)_{\rm SR})$ and $(g_2,h_1)(g_1,h_2)\in E((G\diamond H)_{\rm SR})$. Since $(G \diamond H)_{\rm SR} \supseteq C_4$, we have $O_{\rm SR}(G \diamond H)=\mathcal{B}$ by Corollary~\ref{outcome_sr}(c).

(d) Since $H$ has a $\gamma_H$-pair, (1) of Theorem~\ref{mod_prod_dist3} implies $\diam(G \diamond H)=3$.

Let $g_1$ and $g_2$ be distinct vertices in $G$ with $g_1g_2 \notin E(G)$, and $h_1$ and $h_2$ be distinct vertices in $H$ with $h_1\not\in \mathbb{P}(H)$ and $h_2\not\in \mathbb{P}(H)$. Let $X=\{(g_1,h_1), (g_1,h_2),(g_2,h_1), (g_2,h_2)\}$; we show that no vertex in $X$ is a boundary vertex in $G \diamond H$. Since $G$ has no universal vertex, $N_G[g] \neq V(G)$ for each $g\in V(G)$. If $N_H[h]=V(H)$ for some $h\in V(H)$, then $\diam(H)=2$ since $H$ is not a complete graph, but $\diam(H)=2$ contradicts the hypothesis that $H$ has a $\gamma_H$-pair. So, $H$ has no universal vertex. Since neither $G$ nor $H$ has a universal vertex, $G$ has no $\gamma_G$-pair, $h_1\not\in \mathbb{P}(H)$ and $h_2\not\in \mathbb{P}(H)$, Theorem~\ref{mod_prod_dist3} implies that, for each $(g_i, h_j) \in X$ and for each $(g', h')\in V(G \diamond H)$, $d_{G\diamond H}((g_i, h_j), (g', h'))\neq 3$, where $i,j\in[2]$. So, no vertex in $X$ is a boundary vertex.   

Next, we consider pairs of vertices in $X$ that are at distance two apart in $G \diamond H$. Since $g_1g_2 \notin E(G)$, $d_{G\diamond H}((g_1,h_1),(g_2,h_1))=2$ and $d_{G\diamond H}((g_1,h_2),(g_2,h_2))=2$. By Theorem~\ref{SRedges}, $(g_1,h_1)(g_2,h_1)\in E((G\diamond H)_{\rm SR})$ and $(g_1,h_2)(g_2,h_2)\in E((G\diamond H)_{\rm SR})$. 

Now, we consider two cases depending on the adjacency relation between $h_1$ and $h_2$ in $H$. If $h_1h_2\in E(H)$, then $d_{G\diamond H}((g_1,h_1),(g_2,h_2))=2$ and $d_{G\diamond H}((g_1,h_2),(g_2,h_1))=2$, since $g_1g_2 \notin E(G)$; thus, $(g_1,h_1)(g_2,h_2) \in E((G\diamond H)_{\rm SR})$ and $(g_1,h_2)(g_2,h_1)\in E((G\diamond H)_{\rm SR})$ by Theorem~\ref{SRedges}. If $h_1h_2\not\in E(H)$, then $d_{G\diamond H}((g_1,h_1),(g_1,h_2))=2$ and $d_{G\diamond H}((g_2,h_1),(g_2,h_2))=2$; thus, $(g_1,h_1)(g_1,h_2)\in E((G\diamond H)_{\rm SR})$ and $(g_2,h_1)(g_2,h_2) \in E((G\diamond H)_{\rm SR})$ by Theorem~\ref{SRedges}. 

In each case, $(G \diamond H)_{\rm SR} \supseteq C_4$. Thus, $O_{\rm SR}(G \diamond H)=\mathcal{B}$ by Corollary~\ref{outcome_sr}(c).

(e) We show that $(G \diamond H)_{\rm SR} \supseteq 2C_3$. Let $h_1$ and $h_2$ be distinct adjacent twins in $H$, and both $\{h_1, h\}$ and $\{h_2, h\}$ be $\gamma_H$-pairs for some $h\in V(H)-\{h_1, h_2\}$. Suppose $g_1, g_2\in V(G)$ with $g_1\neq g_2$. From (\ref{eq1}) of Theorem~\ref{mod_prod_dist3}, we have $d_{G\diamond H}((g_1,h),(g_1,h_1))=3$, $d_{G\diamond H}((g_1,h),(g_1,h_2))=3$, $d_{G\diamond H}((g_2,h),(g_2,h_1))=3$ and $d_{G\diamond H}((g_2,h),(g_2,h_2))=3$; thus, for each $i\in[2]$, $(g_i, h)(g_i, h_1) \in E((G \diamond H)_{\rm SR})$ and $(g_i, h)(g_i, h_2) \in E((G \diamond H)_{\rm SR})$ by Theorem~\ref{SRedges}. Since $h_1$ and $h_2$ are distinct adjacent twins in $H$, for each $i\in[2]$, $(g_i, h_1)$ and $(g_i, h_2)$ are distinct adjacent twins in $G\diamond H$ by Theorem~\ref{twins1}(b). By Theorem~\ref{SRedges}, $(g_i, h_1)(g_i, h_2)\in E((G \diamond H)_{\rm SR})$ for each $i\in[2]$. Since $(G \diamond H)_{\rm SR} \supseteq 2C_3$, $O_{\rm SR}(G \diamond H)=\mathcal{B}$ by Corollary~\ref{outcome_sr}(c).~\hfill
\end{proof}

Next, we consider MBSRG played on modular product graphs with $O_{\rm SR} (G \diamond H)\neq \mathcal{B}$. Let $G^-=G-\mathbb{P}(G)$ and $\overline{G}^-=\overline{G}-\mathbb{P}(G)$. We recall the following result on modular product of graphs where one factor has a $\gamma$-pair and the other factor has no universal vertex.

\begin{theorem}\emph{\cite{mod_product}}\label{easy1}
For non-complete graphs $G$ and $H$ such that at least one of the two is not the disjoint union of two cliques, suppose $G\diamond H$ has a $\gamma$-pair in one factor and no universal vertex in the other factor. Then
$$E((G\diamond H)_{\rm SR})={\rm TW}(G\diamond H)\cup E(\overline{G}^-\square\overline{H}^-)\cup E(G^-\times\overline{H}^-)\cup E(\overline{G}^-\times H^-)\cup E(\mathbb{GP}(G)\Box \mathbb{GP}(H))$$
$$\cup \ E(G[V({\rm TW}(G))]\times \mathbb{GP}(H))\cup E(\mathbb{GP}(G)\times H[V({\rm TW}(H))]).$$
\end{theorem}

\begin{theorem}\label{diamGmodH3_theorem3}
Let $G$ and $H$ be connected graphs of order at least two such that neither $G$ nor $H$ is a complete graph. Suppose $\diam(G\diamond H)=3$, and neither $G$ nor $H$ has a pair of distinct adjacent twins. Further, suppose $G$ has neither a universal vertex nor a $\gamma_G$-pair, and $H$ has a $\gamma_H$-pair.
\begin{itemize}
\item[(a)] If $V(H)-\mathbb{P}(H)=\{h^*\}$, then $(G \diamond H)_{\rm SR}=\overline{G} \cup \frac{|V(G)|(|V(H)|-1)}{2}K_2$.
\item[(b)] If $V(H)=\mathbb{P}(H)$, then $(G \diamond H)_{\rm SR}=\frac{|V(G)||V(H)|}{2}K_2$.
\end{itemize}
\end{theorem}

\begin{proof}
Let $G$ and $H$ be graphs as described in the statement of the present theorem. Since neither $G$ nor $H$ has a pair of distinct adjacent twins, $E(G[V({\rm TW}(G))]\times \mathbb{GP}(H))=\emptyset$ and $E(\mathbb{GP}(G)\times H[V({\rm TW}(H))])=\emptyset$. These, together with the assumption that $G$ has no $\gamma_G$-pair, imply that there are no distinct adjacent twins in $G \diamond H$ by Theorem~\ref{twins1}(b). So, ${\rm TW}(G\diamond H)=\emptyset$. Since there's no $\gamma_G$-pair, $\mathbb{P}(G)=\emptyset$; thus, $G^-=G$, $\overline{G}^-=\overline{G}$, and $\mathbb{GP}(G)$ is an edgeless graph with vertex set $V(G)$.

(a) Since $V(H)-\mathbb{P}(H)=\{h^*\}$ and $V(\overline{H})-\mathbb{P}(H)=\{h^*\}$, we have $H^-=H-\mathbb{P}(H)=\{h^*\}$ and $\overline{H}^-=\overline{H}-\mathbb{P}(H)=\{h^*\}$. So, $E(G^-\times\overline{H}^-)=\emptyset$, $E(\overline{G}^-\times H^-)=\emptyset$, and $E(\overline{G}^-\Box\overline{H}^-)=E(\overline{G}^-)=E(\overline{G})$. By Theorem~\ref{easy1}, $E((G \diamond H)_{\rm SR})=E(\overline{G}^- \square \overline{H}^-) \cup E(\mathbb{GP}(G) \square \mathbb{GP}(H))$. For an arbitrary $\gamma_H$-pair $\{h, h'\}$, if $\{h, h''\}$ is also a $\gamma_H$-pair for some $h''\in V(H)-\{h'\}$, then $N_H[h']=N_H[h'']$, which contradicts the hypothesis that $H$ has no distinct adjacent twins. So, $\mathbb{GP}(H)=(\frac{|V(H)|-1}{2})K_2 \cup K_1$, where $h^*$ is the vertex in the $K_1$. Since $\mathbb{GP}(G)$ is an edgeless graph with vertex set $V(G)$, we have $(G \diamond H)_{\rm SR}=\overline{G} \cup \frac{|V(G)|(|V(H)|-1)}{2}K_2$.

(b) Since $V(\overline{H})=V(H)=\mathbb{P}(H)$, we have $H^-=H-\mathbb{P}(H)=\emptyset$ and $\overline{H}^-=\overline{H}-\mathbb{P}(H)=\emptyset$. So, $E(G^-\times\overline{H}^-)=E(\overline{G}^-\times H^-)=\emptyset$ and $E(\overline{G}^-\Box\overline{H}^-)=\emptyset$. By Theorem~\ref{easy1}, $E((G \diamond H)_{\rm SR})=E(\mathbb{GP}(G) \square \mathbb{GP}(H))$. By the same argument used in the proof for (a) of the present theorem, we have $\mathbb{GP}(H)=\frac{|V(H)|}{2}K_2$. Since $\mathbb{GP}(G)$ is an edgeless graph with vertex set $V(G)$, we have $(G \diamond H)_{\rm SR}=\frac{|V(G)||V(H)|}{2}K_2$.~\hfill 
\end{proof}

\begin{corollary}
Let $G$ and $H$ be connected graphs of order at least two such that neither $G$ nor $H$ is a complete graph. If $\diam(G\diamond H)=3$, then $O_{\rm SR}(G \diamond H) \in \{\mathcal{M}, \mathcal{N}, \mathcal{B}\}$.
\end{corollary}

\begin{proof}
Let $G$ and $H$ be graphs described in the statement satisfying $\diam(G\diamond H)=3$. We provide modular product graphs $G\diamond H$ such that $O_{\rm SR}(G \diamond H)=\mathcal{M}$, $O_{\rm SR}(G \diamond H)=\mathcal{N}$ and $O_{\rm SR}(G \diamond H)=\mathcal{B}$, respectively.

First, let $G=C_4$ and let $H=C_6$ be given by $w_1, w_2, w_3, w_4, w_5, w_6, w_1$. Note the following:  (i) $G$ has no universal vertex, no distinct adjacent twins, and no $\gamma_G$-pair; (ii) $\{w_1, w_4\}$, $\{w_2, w_5\}$, and $\{w_3, w_6\}$ are $\gamma_H$-pairs, and $H$ has no distinct adjacent twins. Note that $V(H)=\mathbb{P}(H)$. By Theorem~\ref{diamGmodH3_theorem3}(b), $(G \diamond H)_{\rm SR}=12K_2$. By Corollary~\ref{outcome_sr}(a), $O_{\rm SR}(C_4 \diamond C_6)=\mathcal{M}$.

Second, let $G=\overline{P_5}$ and let $H=P_5$ be given by $w_1, w_2, w_3, w_4, w_5$. Note the following:  (i) $G$ has no universal vertex, no distinct adjacent twins, and no $\gamma_G$-pair; (ii) $\{w_1, w_4\}$ and $\{w_2, w_5\}$ are $\gamma_H$-pairs, $w_3$ belongs to no $\gamma_H$-pair, and $H$ has no distinct adjacent twins. Note that $V(H)-\mathbb{P}(H)=\{w_3\}$. By Theorem~\ref{diamGmodH3_theorem3}(a), $(\overline{P_5} \diamond P_5)_{\rm SR}=P_5 \cup 10K_2$. Thus, $O_{\rm SR}(\overline{P_5} \diamond P_5)=\mathcal{N}$ by Corollary~\ref{outcome_sr}(b).

Third, let $G=P_4$ be given by $u_1, u_2, u_3, u_4$ and $H=P_4$ be given by $w_1, w_2, w_3, w_4$. Note that $\{u_1, u_4\}$ is a $\gamma_G$-pair and $\{w_1, w_4\}$ is a $\gamma_H$-pair. By Theorem~\ref{diamGmodH3_theorem1}(a), $O_{\rm SR}(P_4 \diamond P_4)=\mathcal{B}$.~\hfill
\end{proof}

\textbf{Acknowledgement.} This research was partially supported by US-Slovenia Bilateral Collaboration Grant (BI-US/22-24-101).


\end{document}